\documentclass[11pt, a4paper]{amsart}

\usepackage{a4, a4wide}
\usepackage{amssymb}
\usepackage{amsmath}
\usepackage{amsthm}
\usepackage{amstext}
\usepackage{amscd}
\usepackage{latexsym}
\usepackage{graphics}
\usepackage{color}
\usepackage{enumitem}
\usepackage[all]{xy}
\usepackage[textsize=scriptsize]{todonotes}

\usepackage[colorlinks, pagebackref]{hyperref}
\hypersetup{
  colorlinks=true,
  citecolor=blue,
  linkcolor=blue,
  urlcolor=blue}

\makeatletter
\def\@tocline#1#2#3#4#5#6#7{\relax
  \ifnum #1>\c@tocdepth 
  \else
    \par \addpenalty\@secpenalty\addvspace{#2}%
    \begingroup \hyphenpenalty\@M
    \@ifempty{#4}{%
      \@tempdima\csname r@tocindent\number#1\endcsname\relax
    }{%
      \@tempdima#4\relax
    }%
    \parindent\z@ \leftskip#3\relax \advance\leftskip\@tempdima\relax
    \rightskip\@pnumwidth plus4em \parfillskip-\@pnumwidth
    #5\leavevmode\hskip-\@tempdima
      \ifcase #1
      \or\or \hskip 2em \or \hskip 2em \else \hskip 3em \fi%
      #6\nobreak\relax
    \dotfill\hbox to\@pnumwidth{\@tocpagenum{#7}}\par
    \nobreak
    \endgroup
  \fi}
\makeatother

\theoremstyle{plain}

\newtheorem{theorem}{Theorem}[section]

\newtheorem{lemma}[theorem]{Lemma}
\newtheorem{corollary}[theorem]{Corollary}
\newtheorem{proposition}[theorem]{Proposition}

\theoremstyle{definition}

\newtheorem{remark}[theorem]{Remark}

\numberwithin{equation}{section}

\newcommand{\Ker}{{\rm Ker}}
\newcommand{\coker}{{\rm Coker}}

\newcommand{\Hom}{{\rm Hom}}

\newcommand{\im}{{\rm Im}}

\newcommand{\A}{{\mathbb A}}

\def\<{\langle}
\def\>{\rangle} 
\def\-{\overline} 
\def\~{\widetilde}
\def\^{\widehat}

\def\@{\mathcal}
\def\!{\mathscr}
\def\#{\mathbb}
\def\_{\underline}

\input{xy}
\xyoption{all}

\begin{document}

\title[Strong $\A^1$-invariance]{Corrigendum: Strong $\A^1$-invariance of $\A^1$-connected components of reductive algebraic groups (J. Topol. 16 (2023), no. 2, 634--649.)}

\author{Chetan Balwe}
\address{Department of Mathematical Sciences, Indian Institute of Science Education and Research Mohali, Knowledge City, Sector-81, Mohali 140306, India.}
\email{cbalwe@iisermohali.ac.in}

\author{Amit Hogadi}
\address{Department of Mathematical Sciences, Indian Institute of Science Education and Research Pune, Dr. Homi Bhabha Road, Pashan, Pune 411008, India.}
\email{amit@iiserpune.ac.in}

\author{Anand Sawant}
\address{School of Mathematics, Tata Institute of Fundamental Research, Homi Bhabha Road, Colaba, Mumbai 400005, India.}
\email{asawant@math.tifr.res.in}
\date{}
\thanks{The authors acknowledge the support of India DST-DFG Project on Motivic Algebraic Topology DST/IBCD/GERMANY/DFG/2021/1, SERB MATRICS Grant MTR/2023/000228 and the Department of Atomic Energy, Government of India, under project no. 12-R\&D-TFR-5.01-0500.}

\begin{abstract} 
The proof of \cite[Lemma 5.1]{BHS} is incomplete as it relies on some results in \cite{Choudhury-Hogadi}, the proof of which contains a gap.  The goal of this note is to give a complete and self-contained proof of \cite[Lemma 5.1]{BHS}.  
\end{abstract}

\maketitle

\setlength{\parskip}{4pt plus1pt minus1pt}
\raggedbottom

\section{Introduction}

The main theorem of \cite{BHS} is as follows: 

\begin{theorem}\cite[Theorem 1.1, Remark 1.2]{BHS}
\label{thm intro main}
If $G$ is an algebraic group over a perfect field $k$, then $\pi_0^{\#A^1}(G)$ is a strongly $\#A^1$-invariant sheaf. Consequently, for any Nisnevich locally
trivial $G$-torsor $\@P \to \@X$ in $\@H(k)$,
\[
P \to X \to BG
\]
is an $\#A^1$-fiber sequence.
\end{theorem}

One of the key steps in the proof of Theorem \ref{thm intro main} is \cite[Lemma 5.1]{BHS}.  The proof of \cite[Lemma 5.1]{BHS} given in \cite{BHS} is incomplete for two reasons -
\begin{enumerate}[label=$(\arabic*)$]
\item it relies on a result in \cite{Choudhury-Hogadi}, the proof of which contains a gap, and 
\item the existence of the key exact sequence at the bottom of \cite[page 646]{BHS} (in particular, the application of \cite[Theorem 6.50]{Morel-book}) is not properly justified.
\end{enumerate}   
In this note, we give a complete and self-contained proof of \cite[Lemma 5.1]{BHS}, thereby completing the proof of Theorem \ref{thm intro main}.   

We briefly outline the contents of this note.  Section \ref{section preliminaries} contains preliminaries on long exact sequences of simplicial homotopy groups associated with homotopy principal fibrations of simplicial sheaves.  In Section \ref{section key LES}, we establish the key exact sequence mentioned in $(2)$ above.  In Section \ref{section main result}, we use the results of Section \ref{section key LES} along with an argument circumventing the use of results from \cite{Choudhury-Hogadi} to give a proof of \cite[Lemma 5.1]{BHS} (see Lemma \ref{lemma isogeny}).

\subsection*{Notations and conventions}

We will follow the notation used in \cite{Morel-Voevodsky}, \cite{Morel-book} and \cite{BHS}.  The letter $k$ will always denote a perfect field.  Let $\@H_s(k)$ and $\@H(k)$ denote the pointed unstable simplicial and $\#A^1$-homotopy categories introduced by Morel and Voevodsky in \cite{Morel-Voevodsky}.  These are obtained as homotopy categories associated with the locally injective model structrure on the category $\Delta^{op}Shv(Sm/k)_{\rm Nis}$ of simplicial Nisnevich sheaves of sets on the category $Sm/k$ of essentially smooth, separated schemes on $k$ and its Bousfield localization at the class of projection maps $\@X \times \#A^1 \to \@X$, respectively.  For each $n \in \#N$ and every pointed object $\@X \in \Delta^{op}Shv(Sm/k)_{\rm Nis}$, the simplicial homotopy sheaf $\pi_n(\@X)$ is defined to be the Nisnevich sheafification of the presheaf on $Sm/k$
\[
U \mapsto \Hom_{\@H_s(k)}(\Sigma^nU_+, \@X)
\]
and the $\A^1$-homotopy sheaf $\pi_n^{\#A^1}(\@X)$ is defined to be the Nisnevich sheafification of the presheaf on $Sm/k$
\[
U \mapsto \Hom_{\@H(k)}(\Sigma^nU_+, \@X),
\]
where we suppress the basepoint whenever it is clear from the context.

\section{Principal fibrations of simplicial sheaves on a Grothendieck site}
\label{section preliminaries}

In this section, we collect preliminaries regarding the long exact sequence of homotopy groups associated with principal fibrations of simplicial sheaves that will be used in the proof of the main result.  The key result here is Proposition \ref{prop comparing quotients}, which allows us to compare homotopy quotients in the categories of Nisnevich and \'etale simplicial sheaves. The rest of the material in this section is expected to be well-known to experts and is included only for the convenience of the reader.  

Let $T$ be a Grothendieck site with enough points. We work with the category $\Delta^{op}Shv(T)$ of simplicial sheaves on $T$ endowed with the locally injective model structure.  We will denote by $\ast$ the final object of $T$.

\begin{proposition}
\label{prop local fibration}
Let $\@P \to \@X$ be a local fibration and $\@Q \to \@X$ be any morphism in $\Delta^{op}Shv(T)$.  Then we have a weak equivalence 
\[
\@P \times_{\@X} \@Q \cong \@P \times^h_{\@X} \@Q 
\]
between the fiber product and the homotopy fiber product of $\@P$ and $\@Q$ over $\@X$.
\end{proposition}
\begin{proof}
We factorize $\@Q \to \@X$ as $\@Q \to \@Q' \to \@X$ where $\@Q \to \@Q'$ is a weak equivalence and $\@Q' \to \@X$ is a fibration. We then have the diagram of cartesian squares
\[
\xymatrix{
\@P \times_{\@X} \@Q \ar[r] \ar[d] & \@P \times_{\@X}\@Q' \ar[r] \ar[d] & \@P \ar[d] \\
\@Q \ar[r] & \@Q' \ar[r] & \@X
}
\] 
where $\@P \times_{\@X} \@Q'$ is (weakly equivalent to) the homotopy fibre product $\@P \times^h_{\@X} \@Q$. We wish to prove that the top left arrow is a weak equivalence. It is enough to prove this on stalks. 

Let $x$ denote a point of the site. We wish to prove that 
\[
(\@P \times_{\@X} \@Q)_x \to (\@P \times_{\@X}\@Q')_x 
\]
is a weak equivalence. Since filtered colimits commute with finite limits in the category of sheaves of sets (and hence also in the category of sheaves of simplicial sets), this is equivalent to proving that
\[
\@P_x \times_{\@X_x} \@Q_x \to \@P_x \times_{\@X_x}\@Q'_x 
\]
is a weak equivalence. This is true since $P_x \to X_x$ is a fibration, $\@Q_x \to \@Q'_x$ is a weak equivalence and since the model category of simplicial sets is right proper. 
\end{proof}

We briefly recall the terminology related to group actions (see \cite[page 128]{Morel-Voevodsky} and \cite[page 170]{Morel-book} for details):

Let $\@G$ be a simplicial group sheaf. Recall that a \emph{principal $\@G$-fibration} (or a \emph{principal $\@G$-bundle}) is a morphism $\@P \to \@X$ along with a free $\@G$-action on $\@P$ over $\@X$ such that the induced morphism $\@P/\@G \to \@X$ is an isomorphism. 

More generally, if $\@P$ is a simplicial sheaf with a $\@G$-action, we consider the diagonal action of $\@G$ on $E\@G \times \@P$, which is easily seen to be free. We will refer to the quotient $(E\@G \times \@P)/\@G$ as the \emph{homotopy quotient} of the $\@G$-action on $\@P$.   A morphism $\@P \to \@X$, along with a $\@G$-action on $\@P$ over $\@X$, is said to be a \emph{homotopy principal $\@G$-fibration} if it induces a weak equivalence between $\@X$ and the homotopy quotient of the $\@G$-action on $\@P$. 

\begin{corollary}
\label{cor principal G-fibration}
Let $\@G$ be a simplicial group sheaf and let $\@P \to \@X$ be principal $\@G$-fibration in $\Delta^{op}Shv(T)$. Then, for any global point $x: \ast \to \@X$, we have a weak equivalence
\[
\@P \times_{\@X,x} * \cong \@P\times_{\@X,x}^h \ast.
\]
\end{corollary}
\begin{proof}
Any principal $\@G$-fibration is a local fibration by \cite[Chapter 5, Corollary 2.6]{Goerss-Jardine}. Thus, the result follows from Proposition \ref{prop local fibration}. 
\end{proof}


We now specialize to the situation to which we wish to apply the above results. Let $k$ be a field and let $Sm/k$ denote the category of smooth, essentially of finite type, separated schemes over $k$.  Let $\pi: (Sm/k)_{\text{\'et}} \to (Sm/k)_{\rm Nis}$ be the obvious morphism of sites. We will work with the categories $\Delta^{op}Shv(Sm/k)_{\text{\'et}}$ and $\Delta^{op}Shv(Sm/k)_{\rm Nis}$ of simplicial \'etale and Nisnevich sheaves on $Sm/k$, respectively, equipped with the locally injective model structure. 

\begin{proposition}
\label{prop comparing quotients}
Let $H$ be an \'etale (simplicially discrete) group sheaf on $Sm/k$.  Let $P$ be a fibrant object in $\Delta^{op}Shv(Sm/k)_{\text{\'et}}$ with an $H$-action. Let $p: P \to Q$ be the homotopy quotient of this action in $\Delta^{op}Shv(Sm/k)_{\rm Nis}$ and let $\~p: P \to \~Q$ be the homotopy quotient in $\Delta^{op}Shv(Sm/k)_{\text{\'et}}$. Let $Q_{\text{\'et}} = \mathbf{R}\pi_*(Q)$. Let $U \in Obj(Sm/k)$ and let $u: U \to \@P$ be any morphism. Then, for any $i \geq 1$, the homomorphism $\pi_i(Q, p \circ u) \to \pi_i(Q_{\rm et}, \~p \circ u)$ is an isomorphism. 
\end{proposition}

\begin{proof}
We will apply the restriction functors $\Delta^{op}Shv(Sm/k)_{\rm Nis} \to \Delta^{op}Shv(Sm/U)_{\rm Nis}$ and $\Delta^{op}Shv(Sm/k)_{\text{\'et}} \to \Delta^{op}Shv(Sm/U)_{\text{\'et}}$. To avoid making the notation cumbersome, we will denote the restricted simplicial sheaves, such as $P|_U$, $Q|_U$, etc. by the symbols $P$, $Q$, etc., respectively. We note that these restriction functors are exact and will thus preserve fibre products. 

Since the point $p \circ u: U \to Q$ admits a lift to $P$, the fibre of $P \to Q$ over $p \circ u$ is isomorphic to $H$. By Corollary \ref{cor principal G-fibration}, this is actually the homotopy fibre over $P \to Q$ over $p \circ u$. Thus, we have the simplicial fibration sequence 
\[
H \to  P \to Q
\]
in $\Delta^{op}Shv(Sm/U)_{\rm Nis}$. We note that the isomorphism of $H$ with the fibre of $P \to Q$ can be chosen in such a way that the morphism $H \to P$ is $H$-equivariant and if $e_H: U \to H$ is the identity element of the group $H(U)$, the composition 
\[
U \stackrel{e_H}{\to} H \to P
\]
is equal to $u$. The homotopy fibre of $H \to P$ over $u$ is $\Omega_{p \circ u} Q$. This gives us the simplicial fibration sequence
\begin{equation}
\label{fibration sequence 1}
\Omega_{p \circ u} Q \to H \to P \text{.}
\end{equation}

Similarly, we have the simplicial fibration sequence 
\[
\Omega_{\~p \circ u} \~Q \to H \to P
\]
in $\Delta^{op}Shv(Sm/U)_{\text{\'et}}$. Applying $\mathbf{R}\pi_*$, we get the simplicial fibration sequence 
\begin{equation}
\label{fibration sequence 2}
\Omega_{\~p \circ u} Q_{et} \to H \to P \text{.}
\end{equation}

Note that the morphisms $H \to P$ in both the fibration sequences (\ref{fibration sequence 1}) and (\ref{fibration sequence 2}) are identical since they are both $H$-equivariant and map the identity section to $u$. Comparing the homotopy fibre of $H \to P$ over the point $u: U \to P$ in the two fibration sequences, we see that $\Omega_{p \circ u} Q$ is simplicially equivalent to $\Omega_{\~p \circ u} Q_{et}$. The result immediately follows from this equivalence. 
\end{proof}

\begin{remark}
Actually, the proof of Proposition \ref{prop comparing quotients} shows that for $i>0$, the morphism $Q \to Q_{\text{\'et}}$ induces an isomorphism of groups 
\[
Hom_{\@H_s(k)}(\Sigma^i U_+, Q) \xrightarrow{\sim} Hom_{\@H_s(k)}(\Sigma^i U_+, Q_{\text{\'et}}) 
\]
for any $U \in Obj(Sm/k)$. In other words, the canonical morphism $Q \to Q_{\text{\'et}}$ induces an isomorphism of \emph{presheaves} of $i$-th homotopy groups for $i>0$. We have only stated the result in terms of the associated homotopy group sheaves because that is what we will require in Section \ref{section key LES}.

If we consider the special case $P= \ast$, the homotopy quotient $Q$ is just the classifying space $B H_{\rm Nis}$ (where $H_{\rm Nis} := \pi_*(H)$) and $Q_{\text{\'et}}$ is the \'etale classifying space $B_{\text{\'et}} H$. We recall the following assertion from \cite[page 130]{Morel-Voevodsky}:
\[
\Hom_{\@H_s(k)}(\Sigma^i U_+, B_{\text{\'et}} H) = \begin{cases}
H^1_{\text{\'et}}(U, H), \quad \text{if $i=0$}; \\
H(U), \quad \quad \quad \text{if $i=1$}; \\
0, \quad \quad \quad \quad \quad \text{if $i>1$}
\end{cases}
\]
For $i>0$, this is just the statement that the morphism $BH_{\rm Nis} \to B_{\text{\'et}}H$ induces an isomorphism of $n$-th homotopy group presheaves. Thus, Proposition \ref{prop comparing quotients} can be viewed as a generalization of this statement to homotopy quotients of actions of $H$ on other spaces.
\end{remark}

\section{A key long exact sequence}
\label{section key LES}

Let $\rho: \~G \to G$ be a central isogeny of semisimple algebraic groups with kernel a group $\mu$ of multiplicative type.  The morphism $\~G \to G$ is an \'etale-locally trivial $\mu$-torsor.  Let $\~Q$ and $Q$ denote the \'etale and Nisnevich homotopy quotients of the action of $\~G$ on $G$, respectively.  Let $\pi: \Delta^{op}Shv(Sm/k)_{\text{\'et}} \to \Delta^{op}Shv(Sm/k)_{\rm Nis}$ denote the canonical morphism.  Observe that $G$ is a fibrant object in $\Delta^{op}Shv(Sm/k)_{\text{\'et}}$ and that $\~Q$ and $Q$ can be explicitly described as the quotients
\[
Q = (G \times E\~G)/^{\text{Nis}}\~G
\]
and
\[
\~Q = (G \times E\~G)/^{\text{\'et}}\~G =B \mu
\]
up to weak equivalence.  Note also that 
\[
Q_{\text{\'et}}:= \mathbf{R}\pi_* \~Q = B_{\text{\'et}} \mu.
\]
Thus, there is a canonical morphism $Q \to B_{\text{\'et}} \mu$.

\begin{lemma}
\label{lemma monomorphism}
The morphism $\pi_0(Q) \to \pi_0(B_{\text{\'et}} \mu)$ is a monomorphism. 
\end{lemma}
\begin{proof}
The assertion of the lemma can be verified stalkwise.  Let $U$ be an essentially smooth Henselian local scheme. We will show that $\pi_0(Q)(U) \to \pi_0(B_{\text{\'et}} \mu)(U) \cong H^1_{\text{\'et}}(U, \mu)$ is an injective map. It is easy to check that $\pi_0(Q)(U)$ is the quotient of the action of $\~G(U)$ on $G(U)$.  Let $H$ denote the Galois group of $U^{sh} \to U$, where $U^{sh}$ denotes the strict henselization of $U$.  Consider the short exact sequence
\[
1 \to \mu(U^{sh}) \to \~G(U^{sh}) \to G(U^{sh}) \to 1 
\]
of $H$-modules.  Since $\mu(U^{sh})$ is central in $\~G(U^{sh})$, the long exact sequence of Galois cohomology groups (see \cite[Chapter I, Section 5.7, Proposition 43]{Serre-Galois-cohomology}, for example) takes the form
\[
1 \to \mu(U) \to \~G(U) \to G(U) \to H^1_{\text{\'et}}(U, \mu) \to H^1_{\text{\'et}}(U, \~G) \to \cdots,
\]
where the connecting map $G(U) = G(U^{sh})^H  \to H^1(H, \mu(U^{sh})) =  H^1_{\text{\'et}}(U, \mu)$ is a group homomorphism.  Since $\pi_0(B_{\text{\'et}} \mu)(U) = H^1_{\text{\'et}}(U, \mu)$, the exactness of the above long exact sequence proves the lemma.
\end{proof}

\begin{lemma}
\label{lemma a1 local}
Let $\@X$ be an $\#A^1$-local object in $\Delta^{op}(Sm/k)_{\rm Nis}$. Let $\@F$ be a subsheaf of $\pi_0(\@X)$. Then, the sheaf $\@X \times_{\pi_0(\@X)} \@F$ is $\#A^1$-local. 
\end{lemma}
\begin{proof}
We factor $\@X \to \pi_0(\@X)$ as $\@X \to \^{\@X} \to \pi_0(\@X)$ such that $\^X \to \pi_0(\@X)$ is simplicially fibrant and $\@X \to \^{\@X}$ is a weak equivalence. Then, as $\pi_0(\@X)$ is fibrant, so is $\^{\@X}$. Since $\@F \to \pi_0(\@X)$ is a fibration, the morphism 
\[
\@F \times_{\pi_0(\@X)} \@X \to \@F \times_{\pi_0(\^{\@X})} \^{\@X}
\]
is a weak equivalence (since $\Delta^{op}Shv(Sm/k)_{\rm Nis}$ is right proper). Thus, it suffices to show that $\@F \times_{\pi_0(\^{\@X})} \^{\@X}$ is $\#A^1$-local. 

Let $h: \@Z \times \#A^1 \to \@F \times_{\pi_0(\^{\@X})} \^{\@X}$ be a morphism in the simplicial homotopy category.
\[
\begin{xymatrix}{
\@Z \times \#A^1 \ar[r]^-h &  \@F \times_{\pi_0(\^{\@X})} \^{\@X} \ar[r] \ar[d] & \^{\@X} \ar[d]\\
&\@F \ar[r] & \pi_0(\^{\@X})} 
\end{xymatrix}
\]
Since $\^{\@X}$ is $\#A^1$-local, the two composites
\[
\@Z \rightrightarrows  \@Z \times \#A^1 \to \@F \times_{\pi_0(\@X)} \^{\@X} \to \^{\@X} \to \pi_0(\@X)
\]
are equal, where the leftmost maps are the $0$- and $1$-sections.  Now, any simplicial homotopy $\@Z \times \Delta^1 \to \^{\@X}$ connecting the two maps
\[
\@Z \rightrightarrows  \@Z \times \#A^1 \to \@F \times_{\pi_0(\@X)} \^{\@X} \to \^{\@X}
\]
must factor through $\@F \times_{\pi_0(\@X)} \^{\@X}$, as can be checked stalkwise.  Thus, it follows that $h$ factors through a morphism $\@Z \times \#A^1 \to \@F \times_{\pi_0(\@X)} \^{\@X}$ in the simplicial homotopy category.  This proves that $\@F \times_{\pi_0(\^{\@X})} \^{\@X}$ is $\#A^1$-local, as desired. 
\end{proof}

\begin{remark} Another quick way to prove Lemma \ref{lemma a1 local} is to observe that $\@F \to \pi_0(\@X)$ is a homotopy monomorphism, i.e. the projection morphisms $\@F \times_{\pi_0(\@X)}^h \@F \to \@F$ are isomorphisms. From this, one can immediately see that  $\@X \times_{\pi_0(X)} \@F \to \@X$ is also a homotopy monomorphism. Now, one can easily conclude that if $\@X$ is $\#A^1$-local, then so is $\@X \times_{\pi_0(\@X)} \@F$. (This argument was pointed out by the referee.)
\end{remark}

\begin{proposition}
\label{prop Q a1-local}
Let $\^Q$ denote the fibre product $\pi_0(Q) \times_{\pi_0(B_{\rm et}\mu)} B_{\text{\'et}} \mu$.  Then $\^Q$ is $\#A^1$-local and the canonical morphism $Q \to \^Q$ is a simplicial weak equivalence.
\end{proposition}
\begin{proof}
By \cite[page 137, Proposition 3.1]{Morel-Voevodsky}, $B_{\text{\'et}} \mu$ is $\#A^1$-local. So, by Lemma \ref{lemma a1 local}, it follows that $\^Q$ is $\#A^1$-local.  The morphism $Q \to \^Q$ induces an isomorphism on $\pi_0$ by construction. Since the morphism $\^Q \to B_{\text{\'et}} \mu$ is a pullback of the morphism  $\pi_0(Q) \to \pi_0(B_{\text{\'et}} \mu)$, it induces an isomorphism on $\pi_i$ for all $i>0$, and for every basepoint $u: U \to \^Q$ where $U \in Sm/k$. The morphism $Q \to B_{\text{\'et}} \mu$ induces an isomorphism on $\pi_i$ for all $i>0$ due to Proposition \ref{prop comparing quotients}. Also, the morphism $Q \to Q_{\text{\'et}}$ factors through $Q \to \^Q$.  Thus, we see that $Q \to \^Q$ is a simplicial weak equivalence.
\end{proof}

\begin{corollary}
\label{cor key LES}
Let the notation be as above.  If $\pi_0^{\A^1}(\~G)$ is strongly $\#A^1$-invariant, then we have a long exact sequence
\[
\cdots \to \pi_1^{\A^1}(B_{\text{\'et}} \mu)\to \pi_0^{\A^1}(\~G) \to \pi_0^{\A^1}(G) \to \pi_0^{\A^1}(B_{\text{\'et}} \mu).
\]
\end{corollary}
\begin{proof}
We have a simplicial homotopy principal $\~G$-fibration
\[
\~G \to G \to Q,
\]
which is an $\A^1$-fiber sequence by \cite[Theorem 6.50]{Morel-book}.  This gives rise to an exact sequence of sheaves of groups/sets
\begin{equation}
\label{eqn long exact seq}
\cdots \to \pi_1^{\A^1}(Q) \to \pi_0^{\A^1}(\~G) \to \pi_0^{\A^1}(G) \to \pi_0^{\A^1}(Q).
\end{equation}
In this sequence, the terms $\pi_i^{\#A^1}(Q)$ for $i>1$ may be replaced by $\pi_i^{\#A^1}(B_{\text{\'et}} \mu)$ due to Propositions \ref{prop comparing quotients} and \ref{prop Q a1-local}. The term $\pi_0^{\#A^1}(Q)$ may be replaced by $\pi_0^{\#A^1}(B_{\text{\'et}} \mu)$ due to Lemma \ref{lemma monomorphism} and Proposition \ref{prop Q a1-local}. This gives us the desired long exact sequence. 
\end{proof}

\section{Proof of the main result}
\label{section main result}

The proof of \cite[Lemma 5.1]{BHS} depends upon \cite[Theorem 1.3 and Theorem 1.5]{Choudhury-Hogadi}.  However, the proof of \cite[Lemma 2.9]{Choudhury-Hogadi} contains a gap, which renders the proofs of \cite[Theorem 1.3 and Theorem 1.5]{Choudhury-Hogadi} incomplete in their generality, as of now.  In this note, we give a complete proof of \cite[Lemma 5.1]{BHS} bypassing the use of \cite[Lemma 2.9]{Choudhury-Hogadi}.  Our proof is based on the following weaker version of \cite[Theorem 1.5]{Choudhury-Hogadi}.

\begin{lemma}
\label{lemma correctedch}
Let $k$ be a perfect field.  Let $G$ be a strongly $\A^1$-invariant sheaf of groups on $Sm/k$ and $G \xrightarrow{\phi} H$ an epimorphism of sheaves. Further assume that  $\Ker(\phi)$ is central in $G$. Then $H$ is strongly $\A^1$-invariant if and only if it is $\A^1$-invariant.
\end{lemma}
\begin{proof}
We only need to show that if $H$ is $\A^1$-invariant, then it is strongly $\A^1$-invariant as the reverse implication is trivial.  Let $K:= \Ker(\phi)$. 	For every smooth $k$-scheme $U$, the short exact sequence of Nisnevich sheaves of groups
\[
1 \to K \to G \xrightarrow{\phi} H \to 1
\]
gives us an exact sequence of pointed cohomology sets
\[
1 \to H^0(U,K) \to H^0(U, G) \to H^0(U, H) \to H^1(U, K) \to H^1(U, G) \to H^1(U, H),
\]
by \cite[Ch. III, Proposition 3.3.1]{Giraud}.  Since this exact sequence is functorial in $U$ and since $G$ is strongly $\A^1$-invariant, it follows that $\A^1$-invariance of $H$ is equivalent to strong $\A^1$-invariance of $K$.  Since $K$ is a Nisnevich sheaf of abelian groups and $H$ is $\A^1$-invariant by hypothesis, it follows from \cite[Theorem 5.46]{Morel-book} that $K$ is strictly $\A^1$-invariant.

In order to show strong $\A^1$-invariance of $H$, it is enough to show that for every essentially smooth $k$-scheme $U$, the projection map $U \times \A^1 \to U$ induces an isomorphism $H^1(U, H) \to H^1(U\times \A^1_k, H)$.  Since $K$ is central, by functoriality and \cite[Ch. IV, Remarque 4.2.10]{Giraud}, we have a commutative diagram 
\[
\begin{xymatrix}{
\cdots \ar[r] &  H^1(U, G) \ar[r] \ar[d] & H^1(U, H) \ar[r] \ar[d] & H^2(U, K) \ar[d] \\
\cdots \ar[r] &  H^1(U\times \A^1_k, G) \ar[r] & H^1(U\times \A^1_k, H) \ar[r] & H^2(U\times \A^1_k, K)
}
\end{xymatrix}
\]
whose rows are exact sequences (of pointed sets).  Since $G$ is strongly $\A^1$-invariant and $K$ is strictly $\A^1$-invariant, the desired assertion follows.
\end{proof}

We are now set to give a proof of \cite[Lemma 5.1]{BHS}.  We restate it here for the convenience of the reader.

\begin{lemma}\cite[Lemma 5.1]{BHS}
\label{lemma isogeny}
Let $k$ be a field.  Let $\~G \to G$ be a central isogeny of semisimple algebraic groups, the kernel of which is a group $\mu$ of multiplicative type.  If $\pi_0^{\A^1}(\~G)$ is strongly $\A^1$-invariant, then so is $\pi_0^{\A^1}(G)$.
\end{lemma}
\begin{proof} 
Applying Corollary \ref{cor key LES}, we obtain an exact sequence of sheaves of sets
\begin{equation}
\label{eqn long exact}
\cdots \to \pi_1^{\A^1}(B_{\text{\'et}} \mu)\to \pi_0^{\A^1}(\~G) \to \pi_0^{\A^1}(G) \to \pi_0^{\A^1}(B_{\text{\'et}} \mu).
\end{equation}
From the above exact sequence, we extract the short exact sequence 
\[
1 \to \im\left(\pi_0^{\A^1}(\~G) \to \pi_0^{\A^1}(G)\right) \to \pi_0^{\A^1}(G) \to \im\left(\pi_0^{\A^1}(G) \to \pi_0^{\A^1}(B_{\text{\'et}}\mu)\right) \to 1.
\]
It thus suffices to show that $\im\left(\pi_0^{\A^1}(\~G) \to \pi_0^{\A^1}(G)\right)$ and $\im\left(\pi_0^{\A^1}(G) \to \pi_0^{\A^1}(B_{\text{\'et}}\mu)\right)$ are both strongly $\A^1$-invariant.

\noindent \underline{\emph{Proof of strong $\A^1$-invariance of $\im\left(\pi_0^{\A^1}(G) \to \pi_0^{\A^1}(B_{\text{\'et}}\mu)\right)$}:}

\noindent First note that the sheaf $\pi_0^{\A^1}(B_{\text{\'et}}\mu) = \mathcal H^1_{\text{\'et}}(\mu)$ is strongly $\A^1$-invariant by \cite[Example 4.6]{Asok-crelle}.  The short exact sequence
\[
1 \to \im\left(\pi_0^{\A^1}(G) \to \pi_0^{\A^1}(B_{\text{\'et}}\mu)\right) \to \pi_0^{\A^1}(B_{\text{\'et}}\mu) \to {\coker} \left(\pi_0^{\A^1}(G) \to \pi_0^{\A^1}(B_{\text{\'et}}\mu)\right) \to 1
\]
shows that it suffices to prove $\A^1$-invariance of ${\coker} \left(\pi_0^{\A^1}(G) \to \pi_0^{\A^1}(B_{\text{\'et}}\mu)\right)$.
%
%
For every smooth henselian local $k$-scheme $U$, we have a commutative diagram
\[
\begin{xymatrix}{
G(U) \ar[r] \ar@{->>}[d] & H^1_{\text{\'et}}(U, \mu) \ar[r] \ar[d]^-{\simeq} & H^1_{\text{\'et}}(U, \~G) \ar[r] \ar[d]^-{\simeq} & H^1_{\text{\'et}}(U, G) \\
\pi_0^{\A^1}(G)(U) \ar[r]  & \pi_0^{\A^1}(B_{\text{\'et}}\mu)(U) \ar[r]  & \pi_0^{\A^1}(B_{\text{\'et}}\~G)(U) & } 
\end{xymatrix}
\]
in which the top row is exact.  Note that the cokernel of the map $\pi_0^{\A^1}(G)(U) \to \pi_0^{\A^1}(B_{\text{\'et}}\mu)(U)$ is contained in $\pi_0^{\A^1}(B_{\text{\'et}}\mu)(U) \simeq H^1_{\text{\'et}}(U, \~G)$.  By \cite[Theorem 3.9]{BHS}, the sheaf $\pi_0^{\A^1}(B_{\text{\'et}}\~G) = \mathcal H^1_{\text{\'et}}(\~G)$ is $\A^1$-invariant.  Thus, the cokernel of $\pi_0^{\A^1}(G) \to \pi_0^{\A^1}(B_{\text{\'et}}\mu)$ is an $\A^1$-invariant sheaf as required.

\noindent \underline{\emph{Proof of strong $\A^1$-invariance of $\im\left(\pi_0^{\A^1}(\~G) \to \pi_0^{\A^1}(G)\right)$}:}

\noindent Note that $\pi_0^{\A^1}(G)$ is $\A^1$-invariant by \cite[Corollary 5.2]{Choudhury}.  By the long exact sequence \eqref{eqn long exact} and Lemma \ref{lemma correctedch}, the claim follows if the image of $\pi_1^{\A^1}(B_{\text{\'et}} \mu)\to \pi_0^{\A^1}(\~G)$ is central.  In order to show this, it suffices to show by \cite[Corollary 4.17]{Choudhury} that for all finitely generated field extensions $L/k$, the image of $\pi_1^{\A^1}(B_{\text{\'et}} \mu)(L)\to \pi_0^{\A^1}(\~G)(L)$ is central.  Since $\~G\to \pi_0^{\A^1}(\~G)$ is an epimorphism of Nisnevich sheaves, it suffices to show that for any element $x\in \~G(L)$, the diagram
\begin{equation}\label{bmucentral}
\begin{xymatrix}{
 \pi_1^{\A^1}(B_{\text{\'et}} \mu)(L) \ar[r] \ar@{=}[d] & \pi_0^{\A^1}(\~G)(L)\ar[d]^{\-c_x} \\
 \pi_1^{\A^1}(B_{\text{\'et}} \mu)(L)\ar[r]& \pi_0^{\A^1}(\~G)(L)
 }
\end{xymatrix}
\end{equation}
commutes, where $\-c_x$ denotes the endomorphism of $\pi_0^{\A^1}(\~G)$ induced by the conjugation $\~G \xrightarrow{c_x}\~G$ by $x \in \~G(L)$. If $\-{x}$ denotes the image of $x$ in $G(L)$, then the diagram
\[
\begin{xymatrix}{
 1 \ar[r] &  \mu \ar[r] \ar@{=}[d]  & \~G \ar[r] \ar^{c_x}[d] & G \ar^{c_{\overline{x}}}[d]\\
 1 \ar[r] &  \mu  \ar[r] & \~G \ar[r] & G}
\end{xymatrix}
\]
commutes.  By construction of the fiber sequence $\~G \to G \to B_{\text{\'et}} \mu$, it follows that the diagram
\[
\begin{xymatrix}{
   \~G \ar[r] \ar^{c_x}[d] & G \ar^{c_{\overline{x}}}[d] \ar[r] &  B_{\text{\'et}}\mu  \ar@{=}[d]  \\
  \~G \ar[r] & G \ar[r]  &  B_{\text{\'et}}\mu}
\end{xymatrix}
\]
commutes.  The corresponding commutative diagram with the long exact sequence \eqref{eqn long exact} as rows gives us the required commutativity of \eqref{bmucentral}.
 \end{proof}

\subsection*{Acknowledgements}
The second-named author thanks Tom Bachmann for pointing out a gap in the proof of \cite[Lemma 2.9]{Choudhury-Hogadi}.  The authors thank the anonymous referee for a very careful reading and for some pointed questions, particularly regarding the long exact sequence in Corollary \ref{cor key LES}, which led to substantial improvements to this note.


\begin{thebibliography}{9999}

\bibitem{Asok-crelle}
A. Asok:
\emph{Birational invariants and $\A^1$-connectedness},
J. Reine Angew. Math. 681 (2013), 39--64.

\bibitem{BHS} 
C. Balwe, A. Hogadi, A. Sawant:
\emph{Strong $\A^1$-invariance of $\A^1$-connected components of reductive algebraic groups} J. Topol. 16 (2023), no. 2, 634--649.

\bibitem{Choudhury}
U. Choudhury:
\emph{Connectivity of motivic H-spaces},
Algebr. Geom. Topol. 14 (2014), no. 1, 37--55.


\bibitem{Choudhury-Hogadi}
U. Choudhury, A. Hogadi:
\emph{The Hurewicz map in motivic homotopy theory},
Annals of $K$-theory, Vol. 7 (2022), No. 1, 179--190.


\bibitem{Giraud} 
J. Giraud:
{\it Cohomologie non ab{\'e}lienne}, Die Grundlehren der mathematischen Wissenschaften, Band {\bf 179}, Springer-Verlag, Berlin-New York, 1971. 

\bibitem{Goerss-Jardine} 
P. Goerss, J. Jardine: 
\emph{Simplicial homotopy theory}, Birkhäuser Basel, (2009). 


\bibitem{Morel-book}
F. Morel: 
\emph{$\mathbb A^1$-algebraic topology over a field},
Lecture Notes in Mathematics, Vol. 2052, Springer, Heidelberg, 2012.
 

\bibitem{Morel-Voevodsky}
F. Morel, V. Voevodsky:
\emph{$\A^1$-homotopy theory of schemes},
Inst. Hautes \'Etudes Sci. Publ. Math. 90 (1999), 45--143.

\bibitem{Serre-Galois-cohomology}
J.-P. Serre:
\emph{Galois cohomology},
Springer Monographs in Mathematics, Springer Berlin, \
Heidelberg, 2001.

\end{thebibliography}
\end{document}